\documentclass[12pt, reqno]{amsart}
\usepackage{amsmath, amsthm, amscd, amsfonts, amssymb, graphicx, color}
\usepackage[bookmarksnumbered, colorlinks, plainpages]{hyperref}
\input{mathrsfs.sty}

\newtheorem{theorem}{Theorem}[section]
\newtheorem{lemma}[theorem]{Lemma}
\newtheorem{proposition}[theorem]{Proposition}
\newtheorem{corollary}[theorem]{Corollary}
\theoremstyle{definition}

\newtheorem{remark}[theorem]{Remark}
\numberwithin{equation}{section}

\begin{document}

\title{Extension of Euclidean operator radius inequalities}

\author[M.S. Moslehian, M. Sattari, K. Shebrawi]{Mohammad Sal Moslehian$^1$, Mostafa Sattari$^1$ and Khalid Shebrawi$^2$ }

\address{$^1$ Department of Pure Mathematics, Center Of Excellence in Analysis on Algebraic Structures (CEAAS), Ferdowsi University of Mashhad, P. O. Box 1159, Mashhad 91775, Iran}
\email{moslehian@um.ac.ir}
\email{msattari.b@gmail.com}

\address{$^2$ Department of Applied Sciences, Al-Balqa' Applied University Salt, Jordan; \newline
Department of Mathematics, College of Science, Qassim University, Qassim, Saudi Arabia}
\email{khalid@bau.edu.jo}



\begin{abstract}
To extend the Euclidean operator radius, we define $w_p$ for an $n$-tuples of operators $(T_1,\ldots, T_n)$ in $\mathbb{B}(\mathscr{H})$ by $w_p(T_1,\ldots,T_n):= \sup_{\| x \| =1} \left(\sum_{i=1}^{n}| \langle T_i x, x \rangle |^p \right)^{\frac1p}$ for $p\geq1$.
We generalize some inequalities including Euclidean operator radius of two operators to those involving $w_p$. Further we obtain some lower and upper bounds for $w_p$. Our main result states that if $f$ and $g$ are nonnegative continuous functions on $\left[ 0,\infty \right) $
satisfying $f\left( t\right) g\left(t\right) =t$ for all $t\in \left[ 0,\infty \right) $, then
\begin{equation*}
w_{p}^{rp}\left( A_{1}^{\ast }T_{1}B_{1},\ldots ,A_{n}^{\ast
}T_{n}B_{n}\right) \leq \frac{1}{2}\left\Vert \underset{i=1}{\overset{n}{\sum }}\Big(
\left[ B_{i}^{\ast
}f^{2}\left( \left\vert T_{i}\right\vert \right) B_{i}\right] ^{rp}+\left[
A_{i}^{\ast }g^{2}\left( \left\vert T_{i}^{\ast }\right\vert \right)
A_{i}\right] ^{rp}\Big)\right\Vert
\end{equation*}
for all $p\geq 1$, $r\geq 1$ and operators in $ \mathbb{B}(\mathscr{H})$.
\end{abstract} \maketitle

\section{Introduction}

Let $\mathbb{B}(\mathscr{H})$ be the $C^*$-algebra of all bounded linear operators on a Hilbert space $(\mathscr{H}, \langle \cdot,\cdot\rangle)$. The numerical radius of $A\in \mathbb{B}(\mathscr{H})$ is defined by
\begin{equation*}
w(A)= \sup \{ | \langle A x, x \rangle | : x \in \mathscr{H} , \| x \| =1 \}.
\end{equation*}
It is well known that $w(\cdot)$ defines a norm on $\mathbb{B}(\mathscr{H})$, which is equivalent to the usual operator norm $\| \cdot \|$. Namely, we have
\begin{eqnarray*}
\frac{1}{2} \| A \| \leq w(A) \leq \| A \|.
\end{eqnarray*}
for each $A \in \mathbb{B}(\mathscr{H})$. It is known that if $A\in \mathbb{B}(\mathscr{H})$ is self-adjoint, then $w(A)= \| A \|.$
An important inequality for $w(A)$ is the power inequality stating that $w(A^n)\leq w^n(A)$ for $n=1,2,\ldots$.
There are many inequalities involving numerical radius; see \cite{2, H1, H2, SAL, Y1, Y2} and references therein.\\
The Euclidean operator radius of an $n$-tuple $(T_1,\ldots,T_n)\in \mathbb{B}(\mathscr{H})^{(n)}:= \mathbb{B}(\mathscr{H})\times \ldots \times \mathbb{B}(\mathscr{H})$ was defined in \cite{6} by
\begin{equation*}
w_e(T_1,\ldots,T_n):= \sup_{\| x \| =1} \left(\sum_{i=1}^{n}| \langle T_i x, x \rangle |^2 \right)^{\frac12}\,.
\end{equation*}
The particular cases $n=1$ and $n=2$ are numerical radius and Euclidean operator radius. Some interesting properties of this radius were obtained in \cite{6}. For example, it is established that
\begin{eqnarray} \label{1.1}
\frac{1}{2\sqrt{n}}\left\|\sum_{i=1}^{n}T_iT_i^* \right\|^\frac12\leq w_e(T_1,\cdots,T_n)\leq \left\|\sum_{i=1}^{n}T_iT_i^* \right\|^\frac12.
\end{eqnarray}
We also observe that if $ A =B+ iC$ is the Cartesian decomposition of $A$, then
\[w_e ^2 (B,C) = \mathop{\sup}_{\| x\| =1} \lbrace{|\langle Bx,x \rangle |^2 +|\langle Cx,x \rangle |^2 }\rbrace =\mathop{\sup}_{\| x\| =1} |\langle Ax,x \rangle |^2 = w^2 (A).\]
By the above inequality and $A^*A+AA^* =2(B^2+C^2)$, we have
\[\frac{1}{16} \|A^*A +AA^*\| \leq w^2 (A) \leq \frac{1}{2} \|A^*A +AA^*\|.\]
We define $w_p$ for $n$-tuples of operators $(T_1,\ldots, T_n)\in \mathbb{B}(\mathscr{H})^{(n)}$ for $p\geq1$ by
\begin{equation*}
w_p(T_1,\ldots,T_n):= \sup_{\| x \| =1} \left(\sum_{i=1}^{n}| \langle T_i x, x \rangle |^p \right)^{\frac1p}\,.
\end{equation*}
It follows from Minkowski's inequality for two vectors $a=(a_1, a_2)$ and $b=(b_1, b_2)$, namely,
$$\left(|a_1+b_1|^p+|a_2+b_2|^p\right)^\frac 1p\leq\left(|a_1|^p+|a_2|^p\right)^\frac 1p+
\left(|b_1|^p+|b_2|^p\right)^\frac 1p \quad\text{for}\quad p>1$$
that $w_p$ is a norm.

Moreover $w_{p},p\geq 1,$ for $n$-tuple of operators $\left( T_{1},\ldots
,T_{n}\right) \in \mathbb{B}(\mathscr{H})^{(n)}$ satisfies the
following properties:

(i) $w_{p}\left( T_{1},\ldots ,T_{n}\right) =0\Leftrightarrow T_{1}=\ldots
=T_{n}=0.$

(ii) $w_{p}\left( \lambda T_{1},\ldots ,\lambda T_{n}\right) =\left\vert
\lambda \right\vert w_{p}\left( T_{1},\ldots ,T_{n}\right) $ for all $%
\lambda \in
%TCIMACRO{\U{2102} }%
%BeginExpansion
\mathbb{C}
%EndExpansion
.$

(iii) $w_{p}\left( T_{1}+T_{1}%
%TCIMACRO{\U{b4}}%
%BeginExpansion
{\acute{}}%
%EndExpansion
,\ldots ,T_{n}+T_{n}%
%TCIMACRO{\U{b4}}%
%BeginExpansion
{\acute{}}%
%EndExpansion
\right) \leq w_{p}\left( T_{1},\ldots ,T_{n}\right) +w_{p}\left( T_{1}%
%TCIMACRO{\U{b4}}%
%BeginExpansion
{\acute{}}%
%EndExpansion
,\ldots ,T_{n}%
%TCIMACRO{\U{b4}}%
%BeginExpansion
{\acute{}}%
%EndExpansion
\right) $ for $\left( T_{1}%
%TCIMACRO{\U{b4}}%
%BeginExpansion
{\acute{}}%
%EndExpansion
,\ldots ,T_{n}%
%TCIMACRO{\U{b4}}%
%BeginExpansion
{\acute{}}%
%EndExpansion
\right) \in \mathbb{B}(\mathscr{H})^{(n)}.$

(iv) $w_{p}\left( X^{\ast }T_{1}X,\ldots ,X^{\ast }T_{n}X\right) \leq \left\Vert
X\right\Vert ^{2}w_{p}\left( T_{1},\ldots ,T_{n}\right) $ for $X\in \mathbb{B}(\mathscr{H}).$

Dragomir \cite{1} obtained some inequalities for the Euclidean operator radius $w_e(B,C)={\sup}_{\|x \|=1}\left(|\langle Bx ,x \rangle |^2 +|\langle Cx ,x \rangle |^2\right)^{\frac{1}{2}}$
 of two bounded linear operators in a Hilbert space. In section \ref{sec1} of this paper we extend some his results including inequalities for the Euclidean operator radius of linear operators to $w_p\,\,(p\geq1) $. In addition, we apply some known inequalities for getting new inequalities for $w_p$ in two operators.\\
In section \ref{sec2} we prove inequalities for $w_p$ for $n$-tuples of operators. Some of our result in this section, generalize some inequalities in section \ref{sec1}. Further, we find some lower and upper bounds for $w_p$.

%=================================================================================================%

\section{Inequalities for $w_p$ for two operators}
\label{sec1}
To prove our generalized numerical radius inequalities, we need several known lemmas. The first lemma is a simple result of the classical Jensen inequality and a generalized mixed Cauchy--Schwarz inequality \cite{5, MOS,Kit}.
%=================================================================================================%
\begin{lemma} \label{le1}
 For $ a,b \geq0$, $0 \leq \alpha \leq 1$ and $r\neq 0$,
\begin{enumerate}
\item[(a)]
$ \displaystyle
a^\alpha b^{1-\alpha} \leq \alpha a +(1-\alpha)b \leq
\left[\alpha a^r +(1-\alpha) b^r\right]^{\frac{1}{r}} \quad\text{for}\quad r\geq 1$,
\item[(b)]
If $\displaystyle
 A \in \mathbb{B}(\mathscr{H}), then~
| \langle Ax, y \rangle | ^2 \leq \langle | A | ^{2 \alpha } x,x \rangle \langle | A^*| ^{2(1-\alpha) } y,y \rangle$ for all $x,y\in \mathscr{H},$ where $|A|=(A^*A)^\frac12.$
\item[(c)]
Let $A\in \mathbb{B}(\mathscr{H}) $, and $f$ and $g$ be
nonnegative continuous functions on $\left[ 0,\infty \right) $ satisfying $f\left( t\right) g\left( t\right) =t$ for all $t\in %
\left[ 0,\infty \right) $. Then%
\begin{equation*}
\left\vert \left\langle Ax,y\right\rangle \right\vert \leq \left\Vert
f\left( \left\vert A\right\vert \right) x\right\Vert \left\Vert g\left(
\left\vert A^{\ast }\right\vert \right) y\right\Vert
\end{equation*}%
for all $x, y \in \mathscr{H}$.
\end{enumerate}
\end{lemma}
%=================================================================================================%
\begin{lemma} [\textbf{McCarthy inequality \cite{Yuki}}] \label{le2}
Let $A \in \mathbb{B}(\mathscr{H})$, $ A \geq 0$ and let $x \in \mathscr{H}$ be any unit vector. Then
\begin{enumerate}
\item[(a)]
$\langle A x, x \rangle ^r \leq \langle A^r x, x \rangle \quad \text{for} \quad r \geq 1$,
\item[(b)]
$\langle A^r x , x \rangle \leq \langle Ax, x \rangle ^r \quad \text{for} \quad 0<r \leq 1$.
\end{enumerate}
\end{lemma}
%=================================================================================================%
Inequalities of the following lemma were obtained for the first time by Clarkson\cite{5}.
\begin{lemma} \label{le3}
Let $X$ be a normed space and $x,y\in X$. Then for all $p\geq 2$ with $\frac1p+\frac1q=1$,
\begin{enumerate}
\item[(a)]
$2(\|x\|^p+\|y\|^p)^{q-1} \leq \|x+y\|^q+\|x-y\|^q$,
\item[(b)]
$2(\|x\|^p+\|y\|^p) \leq \|x+y\|^p+\|x-y\|^p\leq 2^{p-1}(\|x\|^p+\|y\|^p)$,
\item[(c)]
$\|x+y\|^p+\|x-y\|^p \leq 2(\|x\|^q+\|y\|^q)^{p-1}.$
\end{enumerate}
If $1<p\leq 2$ the converse inequalities hold.
\end{lemma}
Making the transformations $x\rightarrow \frac{x+y}{2}$ and $y\rightarrow \frac{x-y}{2}$ we observe that inequalities (a) and (c) in Lemma \ref{le3} are equivalent and so are the first and the second inequalities of (b).
%=================================================================================================%
First of all we obtain a relation between $w_p$ and $w_e$ for $p\geq 1$.
\begin{proposition}
Let $B,C\in \mathbb{B}(\mathscr{H}) $. Then%
\begin{equation*}
w_{p}\left( B,C\right) \leq w_{q}\left( B,C\right) \leq 2^{\frac{1}{q}-\frac{%
1}{p}}w_{p}\left( B,C\right)
\end{equation*}%
for $p\geq q\geq 1.$ In particular
\begin{equation}\label{2.1}
w_{p}\left( B,C\right) \leq w_{e}\left( B,C\right) \leq 2^{\frac{1}{2}-\frac{%
1}{p}}w_{p}\left( B,C\right)
\end{equation}%
for $p\geq 2,$ and%
\begin{equation*}
2^{\frac{1}{2}-\frac{1}{p}}w_{p}\left( B,C\right) \leq w_{e}\left(
B,C\right) \leq w_{p}\left( B,C\right)
\end{equation*}%
for $1\leq p\leq 2.$
\end{proposition}
\begin{proof}
An application of Jensen's inequality says that for $a,b>0$ and $p\geq q>0,$ we have%
\begin{equation*}
\left( a^{p}+b^{p}\right) ^{\frac{1}{p}}\leq \left( a^{q}+b^{q}\right) ^{%
\frac{1}{q}}.
\end{equation*}%
Let $x\in \mathscr{H} $ be a unit vector. Choosing $a=\left\vert \left\langle Bx,x\right\rangle \right\vert $ and $%
b=\left\vert \left\langle Cx,x\right\rangle \right\vert $, we have%
\begin{equation*}
\Big( \left\vert \left\langle Bx,x\right\rangle \right\vert ^{p}+\left\vert
\left\langle Cx,x\right\rangle \right\vert ^{p}\Big) ^{\frac{1}{p}}\leq
\Big( \left\vert \left\langle Bx,x\right\rangle \right\vert ^{q}+\left\vert
\left\langle Cx,x\right\rangle \right\vert ^{q}\Big) ^{\frac{1}{q}}.
\end{equation*}%
Now the first inequality follows by taking the supremum over all unit
vectors in $\mathscr{H}$. A simple consequence of the classical Jensen's inequality
concerning the convexity or the concavity of certain power functions says
that for $a,b\geq 0,0\leq \alpha \leq 1$ and $p\geq q$, we have%
\begin{equation*}
\left( \alpha a^{q}+\left( 1-\alpha \right) b^{q}\right) ^{\frac{1}{q}}\leq
\left( \alpha a^{p}+\left( 1-\alpha \right) b^{p}\right) ^{\frac{1}{p}}.
\end{equation*}%
For $\alpha =\frac{1}{2}$, we get%
\begin{equation*}
\left( a^{q}+b^{q}\right) ^{\frac{1}{q}}\leq 2^{\frac{1}{q}-\frac{1}{p}%
}\left( a^{p}+b^{p}\right) ^{\frac{1}{p}}.
\end{equation*}%
Again let $x\in \mathscr{H}$ be a unit vector. Choosing $a=\left\vert \left\langle Bx,x\right\rangle \right\vert $ and $%
b=\left\vert \left\langle Cx,x\right\rangle \right\vert $ we get%
\begin{equation*}
\Big( \left\vert \left\langle Bx,x\right\rangle \right\vert ^{q}+\left\vert
\left\langle Cx,x\right\rangle \right\vert ^{q}\Big) ^{\frac{1}{q}}\leq 2^{%
\frac{1}{q}-\frac{1}{p}}\Big( \left\vert \left\langle Bx,x\right\rangle
\right\vert ^{p}+\left\vert \left\langle Cx,x\right\rangle \right\vert
^{p}\Big) ^{\frac{1}{p}}.
\end{equation*}%
Now the second inequality follows by taking the supremum over all unit
vectors in $\mathscr{H}$.
\end{proof}
%=================================================================================================%

On making use of inequality \eqref{2.1} we
find a lower bound for $w_p\,\,(p\geq2).$
\begin{corollary}\label{cor2.5}
If $B,C \in \mathbb{B}(\mathscr{H})$, then for $p\geq 2$
\begin{equation*}
 w_p(B,C)\geq 2^{\frac1p-2} \|B^*B+C^*C\|^\frac12\,.
\end{equation*}
\end{corollary}
\begin{proof}
According to inequalities \eqref{1.1} and \eqref{2.1} we can write
\[
w_e(B,C)\geq \frac{1}{2\sqrt{2}}\|B^*B+C^*C\|^\frac12
\]
and
\[
w_p (B,C)\geq 2^{\frac{1}{p}-\frac{1}{2} } w_e (B,C),
\]
respectively. We therefore get desired inequality.
\end{proof}
%=================================================================================================%

The next result is concerned with some lower bounds for $w_p$. This consequence has several inequalities as special cases. Our result will be generalized to $n$-tuples of operators in the next section.

\begin{proposition}
Let $B,C\in \mathbb{B}(\mathscr{H})$. Then for $p\geq1$
\begin{equation}\label{2.2}
 w_p (B,C) \geq 2^{\frac{1}{p}-1} \max \left( w(B + C) , w( B- C) \right).
 \end{equation}
 This inequality is sharp.
\end{proposition}

\begin{proof}
We use convexity of function $f(t) = t^p\,\,(p\geq1)$ as follows:
\begin{align*}
(|\langle Bx,x \rangle |^p +|\langle Cx,x \rangle | ^p )^{\frac{1}{p}} & \geq 2^{\frac{1}{p} -1} (|\langle Bx,x \rangle | +|\langle Cx,x \rangle |)\\
&\geq 2^{\frac{1}{p} -1} |\langle Bx,x \rangle \pm \langle Cx,x \rangle | \\
&= 2^{\frac{1}{p} -1} |\langle (B \pm C) x,x \rangle |\,.
\end{align*}
Taking supremum over $x\in\mathscr{H}$ with $\|x\|=1$ yields that
\[w_p (B,C) \geq 2^{\frac{1}{p} -1} w(B \pm C).\]
 For sharpness one can obtain the same quantity $2 ^{ \frac{1}{p}} w( B)$ on both sides of the inequality by putting $B=C$.
\end{proof}
%=================================================================================================%
\begin{corollary}\label{cor2.10}
If $A=B+iC $ is the Cartesian decomposition of $A$, then for all $p\geq 2$
$$
 w_p (B , C)\geq 2^{ \frac{1}{p}-1} \max \left( \| B+ C \| , \| B -C \| \right),
$$
and
\[
w(A)\geq 2^{ \frac{1}{p}-2} \max \left( \|(1-i) A+(1+i)A^* \| , \| (1+i) A+(1-i)A^* \| \right)
\]
\end{corollary}
\begin{proof}
Obviously by inequality \eqref{2.2} we have the first inequality. For the second we use inequality \eqref{2.1}.
\end{proof}
%=================================================================================================%
\begin{corollary}
If $B,C\in \mathbb{B}(\mathscr{H})$, then for $p\geq 1$
\begin{equation}\label{2.3}
w_p (B,C) \geq 2^{\frac{1}{p} -1} \max \lbrace w(B), w(C) \rbrace.
\end{equation}
In addition, if $A=B+iC$ is the Cartesian decomposition of $A$, then for $p\geq 2$
\begin{equation*}
w (A) \geq 2^{\frac{1}{p} -2} \max \left( \| A+A^* \| , \| A-A^* \| \right).
\end{equation*}
\end{corollary}
\begin{proof}
By inequality \eqref{2.2} and properties of the numerical radius, we have
\begin{align*}
2w_p (B,C) \geq 2^{\frac{1}{p} -1} (w(B+C) + w(B-C))\geq 2^{\frac{1}{p} -1} w (B+C +B-C)\,.
\end{align*}
So
\[w_p (B,C) \geq 2^{\frac{1}{p} -1} w(B)\,.\]
By symmetry we conclude that
\[w_p (B,C) \geq 2^{\frac{1}{p} -1} \max (w(B) , w(C)). \]
While the second inequality follows easily from inequality \eqref{2.1}.
\end{proof}
%=================================================================================================%
Now we apply part (b) of Lemma \ref{le3} to find some lower and upper bounds for $w_p\,\,( p>1)$.

\begin{proposition} \label{pro2.10}
Let $B, C\in \mathbb{B}(\mathscr{H})$. Then for all $p\geq 2$,\\
\noindent (i) $2^{\frac1p-1}w_p(B+C,B-C)\leq w_p(B,C)\leq 2^{-\frac1p}w_p(B+C,B-C)$;\\
\noindent (ii) $2^{\frac1p-1}\big(w^p(B+C)+w^p(B-C)\big)^\frac 1p\leq w_p(B,C)\leq 2^{-\frac1p}\big(w^p(B+C)+w^p(B-C)\big)^\frac 1p$.\\
If $1<p\leq 2$ these inequalities hold in the opposite direction.
\end{proposition}
\begin{proof}
Let $x\in\mathscr{H}$ be a unit vector. Part (b) of Lemma \ref{le3} implies that for any $p\geq 2$
\[2^{1-p}(|a+b|^p+|a-b|^p)\leq |a|^p+|b|^p\leq \frac12(|a+b|^p+|a-b|^p)\,.\]
Replacing $a=|\langle Bx,x\rangle|$ and $b=|\langle Cx,x\rangle|$ in above inequalities we obtain the desired inequalities.
\end{proof}
%=================================================================================================%
\begin{remark}
In inequality \eqref{2.3}, if we take $B+C$ and $B-C$ instead of $B$ and $C$, then for
$p\geq1$
\[
w_p(B+C, B-C)\geq 2^{\frac1p-1}\max \lbrace w(B+C), w(B-C) \rbrace\,.
\]
By employing the first inequality of part (i) of Proposition \ref{pro2.10}, we get
\[
w_p(B, C)\geq 2^{\frac2p-2}\max \lbrace w(B+C), w(B-C) \rbrace
\]
for $p\geq1$.\\
Taking $B+C$ and $B-C$ instead of $B$ and $C$ in the second inequality of part (ii) of Proposition \ref{pro2.10}, we reach
\[
w_p(B+C, B-C)\leq 2^{1-\frac1p}\left(w^p(B)+w^p(C)\right)^\frac1p\,.
\]
for all $p\geq1$.\\
Now by applying the second inequality of part (i) of Proposition \ref{pro2.10}, we infer for $p\geq1$ that
 \[
 w_p(B,C)\leq2^{1-\frac2p}\left(w^p(B)+w^p(C)\right)^\frac1p.
 \]
 So
 \[
 2^{\frac2p-2}\max \lbrace w(B+C), w(B-C) \rbrace\leq w_p(B,C)\leq2^{1-\frac2p}\left(w^p(B)+w^p(C)\right)^\frac1p.
 \]
 Moreover if $B$ and $C$ are self-adjoint, then
 \[
 2^{\frac2p-2}\max \lbrace \|B+C\|, \|B-C\| \rbrace\leq w_p(B,C)\leq2^{1-\frac2p}\left(\|B\|^p+\|C\|^p\right)^\frac1p
 \]
 for all $p\geq1.$
 \end{remark}
%=================================================================================================%
In the following result we find another lower bound for $w_p\,\,(p\geq 1)$.
\begin{theorem}\label{pro2.12}
Let $B, C\in \mathbb{B}(\mathscr{H})$. Then for $p\geq 1$
\[
 w_p ( B, C)\geq 2 ^{\frac{1}{p}-1} w^{ \frac{ 1}{ 2}}( B ^ 2 + C^ 2) .
\]
\end{theorem}
\begin{proof}
It follows from \eqref{2.2} that
\begin{align*}
 2 ^{ \frac{2}{p} -2} w ^2 ( B \pm C) \leq w ^2_p(B, C)\,.
\end{align*}
Hence
\begin{align*}
2 w ^ 2_p(B, C) &\geq 2 ^{ \frac{2}{p} -2} \big[ w ^2 (B + C) + w^ 2 (B- C) \big] \\
& \geq 2 ^{ \frac{2}{p} -2} \big[ w \left( ( B+ C) ^ 2 \right) + w \left( ( B - C ) ^ 2 \right) \big] \\
& \geq 2 ^{ \frac{2}{p} -2} \big[ w \left( ( B+ C ) ^ 2 + ( B - C) ^ 2 \right) \big] = 2 ^{ \frac{2}{p} -1} w ( B^2 + C^2)\,.
\end{align*}
It follows that
\[w_p(B, C) \geq 2 ^{ \frac{1}{p} -1} w ^{ \frac{ 1}{ 2} }( B ^2+ C^2).\]
\end{proof}
%=================================================================================================%
\begin{corollary}
If $A=B+iC $ is the Cartesian decomposition of $A$ , then
\[
 w_p(B, C)\geq 2 ^{ \frac{1}{p} -1} \| B^2 + C^2 \| ^{ \frac{ 1}{ 2}}.
\]
And
\[
 w (A)\geq 2 ^{ \frac{1}{p} -\frac32} \| A^*A+AA^* \| ^{ \frac{ 1}{ 2}}.
\]
for any $p\geq 2$.
\end{corollary}
\begin{proof}
The first inequality is obvious. For the second we have $A^*A+AA^*=2(B^2+C^2)$. Now by using inequality \eqref{2.1} the proof is complete.
\end{proof}
%=================================================================================================%
\begin{corollary}
If $B, C\in\mathbb{B}(\mathscr{H}) $, then for $p\geq2$
\[
w_p(B,C)\geq 2^{\frac2p-\frac32} w^{\frac12}\left(B^2+C^2\right).
\]
\end{corollary}
\begin{proof}
By choosing $B+C$ and $B-C$ instead of $B$ and $C$ in Theorem \ref{pro2.12} and employing part (i) of Proposition \ref{pro2.10} we conclude that the desired inequality.
\end{proof}
%=================================================================================================%
The following result providing other bound for $w_p\,\,(p > 1)$ may be stated as follows:
\begin{proposition}\label{pro2.15}
Let $B,C \in \mathbb{B}(\mathscr{H})$. Then
\[
w_p(B,C) \leq w _q \left( \frac{ B+C}{2} , \frac{ B-C}{2} \right).
\]
for any $p \geq 2, 1<q\leq 2$ with $\frac{1}{p} +\frac{1}{q}=1$. If $1<p\leq 2,$ the reverse inequality holds.
\end{proposition}
\begin{proof}
Let $x\in \mathscr{H}$ be a unit vector.
Part (a) of Lemma \ref{le3} implies that
\[
|a|^p + |b|^p \leq 2 ^{ \frac{ 1}{ 1-q}} \left( |a+b| ^q +|a- b|^q \right) ^{ \frac{ 1}{ q-1}}\,.
\]
So
\[
(|a| ^p + | b| ^p) ^{\frac{1}{p}} \leq 2 ^{ \frac{1}{p(1-q)}}\left( |a + b | ^q + | a-b| ^q \right) ^{ \frac{ 1}{ p(q-1)}}\,.
\]
Now replacing $a= \langle Bx, x \rangle $ and $b=\langle Cx,x \rangle$ in the above inequality we conclude that
{\small\begin{equation}\label{aba}
\left( | \langle Bx, x \rangle |^p +| \langle Cx,x\rangle |^p \right)^{\frac{1}{p}}\leq \left( \left| \left\langle \left( \frac{B+C}{2} \right) x, x \right\rangle \right| ^q + \left| \left\langle \left( \frac{B-C}{2}\right) x,x \right\rangle \right| ^q \right)^{ \frac{ 1}{q}}.
\end{equation}}
By taking supremum over $x\in \mathscr{H}$ with $\|x \| =1$ we deduce that
\[
w_p(B,C) \leq w _q \left( \frac{ B+C}{2} , \frac{ B-C}{2} \right)
\]
for any $p \geq 2, 1<q\leq 2$ with $\frac{1}{p} +\frac{1}{q}=1$.
\end{proof}
%=================================================================================================%
\begin{corollary}
Inequality \eqref{aba} implies that
\[
w_p(B,C) \leq \left( w ^q \left( \frac{ B+C}{2} \right) + w^q \left( \frac{ B-C}{2}\right) \right) ^{ \frac{1}{q}}.
\]
for any $1<q\leq 2, p \geq 2$ with $\frac{1}{p} +\frac{1}{q}=1$. Further, if $B$and $C$ are self-adjoint, then
\[
w_p(B,C) \leq \frac 12\left( \|B+C\|^q + \|B-C\|^q \right) ^{ \frac{1}{q}}.
\]

If $1<p\leq 2,$ the converse inequalities hold.
\end{corollary}
%=================================================================================================%
\begin{corollary}
If $B,C \in \mathbb{B}(\mathscr{H})$, then
$$w_q\left(\frac{B+C}{2},\frac{B-C}{2}\right)\leq 2^{\frac1p} w_p\left(\frac{B+C}{2},\frac{B-C}{2}\right).$$
for all $1<p \leq2$ with $\frac1p+\frac1q=1$. If $p\geq2$, the above inequality is valid in the opposite direction.
\end{corollary}
\begin{proof}
By Proposition \ref{pro2.15} we have
$$w_q\left(\frac{B+C}{2},\frac{B-C}{2}\right)\leq w_p(B,C).$$
for all $1<p\leq2$ with $\frac1p+\frac1q=1.$
Proposition \ref{pro2.10} follows that
$$w_p(B,C)\leq 2^{\frac1p-1} w_p(B+C,B-C)= 2^{\frac1p} w_p\left(\frac{B+C}{2},\frac{B-C}{2}\right)\,.$$
We therefore get the desired inequality.
\end{proof}
%=================================================================================================%

\section{Inequalities of $w_p$ for $n$-tuples of operators}
\label{sec2}
In this section, we are going to obtain some numerical radius inequalities for $n$-tuples of operators. Some generalization of inequalities in the previous section are also established. According to the definition of numerical radius, we immediately get the following double inequality for $p\geq 1$
\begin{equation*}
w_{p}\left( T_{1},\ldots ,T_{n}\right) \leq \left( \underset{i=1}{\overset{n}%
{\sum }}w^{p}\left( T_{i}\right) \right) ^{\frac{1}{p}}\leq \underset{i=1}{\overset{n}{\sum }%
}w\left( T_{i}\right).
\end{equation*}
%=================================================================================================%
An application of Holder's inequality gives the next result, which is a generalization of inequality \eqref{2.2}.
\begin{theorem}
Let $\left( T_{1},\ldots ,T_{n}\right) \in \mathbb{B}(\mathscr{H}) ^{(n)}$ and $0\leq \alpha _{i}\leq 1$, $i=1,\ldots n,$ with $\underset{%
i=1}{\overset{n}{\sum }}\alpha _{i}=1.$ Then%
\begin{equation*}
w_{p}\left( T_{1},\ldots ,T_{n}\right) \geq w\left( \alpha _{1}^{1-\frac{1}{p%
}}T_{1}\pm \alpha _{2}^{1-\frac{1}{p}}T_{2}\pm \ldots \pm \alpha _{n}^{1-%
\frac{1}{p}}T_{n}\right)
\end{equation*}
for any $p> 1$.
\end{theorem}
\begin{proof}
In the Euclidean space $\mathbb{R}^{n}$ with the standard inner product, Holder's inequality
\begin{equation*}
\underset{i=1}{\overset{n}{\sum }}\left\vert x_{i}y_{i}\right\vert \leq
\left( \underset{i=1}{\overset{n}{\sum }}\left\vert x_{i}\right\vert
^{p}\right) ^{\frac{1}{p}}\left( \underset{i=1}{\overset{n}{\sum }}%
\left\vert y_{i}\right\vert ^{q}\right) ^{\frac{1}{q}}
\end{equation*}%
holds, where $p$ and $q$ are in the open interval $(1,\infty )$ with $\frac{1}{p}+%
\frac{1}{q}=1$ and $\left( x_{1},\ldots ,x_{n}\right)$, $\left( y_{1},\ldots
,y_{n}\right) \in \mathbb{R}^{n}$. For $\left( y_{1},\ldots ,y_{n}\right) =\left( \alpha _{1}^{1-\frac{1%
}{p}},\ldots ,\alpha _{n}^{1-\frac{1}{p}}\right) $ we have%
\begin{equation*}
\underset{i=1}{\overset{n}{\sum }}\left\vert \alpha _{i}^{1-\frac{1}{p}%
}x_{i}\right\vert \leq \left( \underset{i=1}{\overset{n}{\sum }}\left\vert
x_{i}\right\vert ^{p}\right) ^{\frac{1}{p}}\left( \underset{i=1}{\overset{n}{%
\sum }}\left\vert \alpha _{i}^{1-\frac{1}{p}}\right\vert ^{q}\right) ^{\frac{%
1}{q}}.
\end{equation*}%
Thus
\begin{equation*}
\left( \underset{i=1}{\overset{n}{\sum }}\left\vert x_{i}\right\vert
^{p}\right) ^{\frac{1}{p}}\geq \underset{i=1}{\overset{n}{\sum }}\left\vert
\alpha _{i}^{1-\frac{1}{p}}x_{i}\right\vert .
\end{equation*}%
Choosing $x_{i}=\left\vert \left\langle T_{i}x,x\right\rangle \right\vert
,i=1,\ldots n$, we get
\begin{align*}
&\hspace{-1cm}\left( \underset{i=1}{\overset{n}{\sum }}\left\vert \left\langle
T_{i}x,x\right\rangle \right\vert ^{p}\right) ^{\frac{1}{p}}\\
 &\geq \underset
{i=1}{\overset{n}{\sum }}\left\vert \left\langle \alpha _{i}^{1-\frac{1}{p}
}T_{i}x,x\right\rangle \right\vert \\
&\geq \left\vert \left\langle \alpha _{1}^{1-\frac{1}{p}}T_{1}x,x\right
\rangle \pm \left\langle \alpha _{2}^{1-\frac{1}{p}}T_{2}x,x\right\rangle
\pm \ldots \pm \left\langle \alpha _{n}^{1-\frac{1}{p}}T_{n}x,x\right\rangle
\right\vert \\
&=\left\vert \left\langle \left( \alpha _{1}^{1-\frac{1}{p}}T_{1}\pm \alpha
_{2}^{1-\frac{1}{p}}T_{2}\pm \ldots \pm \alpha _{n}^{1-\frac{1}{p}}
T_{n}\right) x,x\right\rangle \right\vert .
\end{align*}
Now the result follows by taking the supremum over all unit vectors in $\mathscr{H}$.
\end{proof}
%=================================================================================================%
Now we give another upper bound for the powers of $w_p$. This result has several inequalities as special cases, which considerably generalize the second inequality of \eqref{1.1}.
\begin{theorem}
Let $\left( T_{1},\ldots ,T_{n}\right) ,\left( A_{1},\ldots ,A_{n}\right)
,\left( B_{1},\ldots ,B_{n}\right) \in \mathbb{B}(\mathscr{H})^{(n)} $
and let $f$ and $g$ be nonnegative continuous functions on $\left[ 0,\infty \right) $
satisfying $f\left( t\right) g\left(
t\right) =t$ for all $t\in \left[ 0,\infty \right) $. Then
{\footnotesize \begin{equation*}
w_{p}^{rp}\left( A_{1}^{\ast }T_{1}B_{1},\ldots ,A_{n}^{\ast
}T_{n}B_{n}\right) \leq \frac{1}{2}\left\Vert \underset{i=1}{\overset{n}{\sum }}\Big(
\left[ B_{i}^{\ast
}f^{2}\left( \left\vert T_{i}\right\vert \right) B_{i}\right] ^{rp}+\left[
A_{i}^{\ast }g^{2}\left( \left\vert T_{i}^{\ast }\right\vert \right)
A_{i}\right] ^{rp}\Big)\right\Vert
\end{equation*}}
for $p\geq 1$ and $r\geq 1.$
\end{theorem}
\begin{proof}
Let $x\in \mathscr{H}$ be a unit vector.
\begin{align*}
&\hspace{-0.3cm}\underset{i=1}{\overset{n}{\sum }}\left\vert \left\langle A_{i}^{\ast
}T_{i}B_{i}x,x\right\rangle \right\vert ^{p}\\
&=\underset{i=1}{\overset{n}{%
\sum }}\left\vert \left\langle T_{i}B_{i}x,A_{i}x\right\rangle \right\vert
^{p} \\
&\leq \underset{i=1}{\overset{n}{\sum }}\left\Vert f\left( \left\vert
T_{i}\right\vert \right) B_{i}x\right\Vert ^{p}\left\Vert g\left( \left\vert
T_{i}^{\ast }\right\vert \right) A_{i}x\right\Vert ^{p}
 \hspace{0.5cm}({\rm by\ Lemma}\ \ref{le1}(c)) \\
&=\underset{i=1}{\overset{n}{\sum }}\left\langle f\left( \left\vert
T_{i}\right\vert \right) B_{i}x,f\left( \left\vert T_{i}\right\vert \right)
B_{i}x\right\rangle ^{\frac{p}{2}}\left\langle g\left( \left\vert
T_{i}^{\ast }\right\vert \right) A_{i}x,g\left( \left\vert T_{i}^{\ast
}\right\vert \right) A_{i}x\right\rangle ^{\frac{p}{2}} \\
&=\underset{i=1}{\overset{n}{\sum }}\left\langle B_{i}^{\ast }f^{2}\left(
\left\vert T_{i}\right\vert \right) B_{i}x,x\right\rangle ^{\frac{p}{2}%
}\left\langle A_{i}^{\ast }g^{2}\left( \left\vert T_{i}^{\ast }\right\vert
\right) A_{i}x,x\right\rangle ^{\frac{p}{2}} \\
&\leq \underset{i=1}{\overset{n}{\sum }}\left\langle \left( B_{i}^{\ast
}f^{2}\left( \left\vert T_{i}\right\vert \right) B_{i}\right)
^{p}x,x\right\rangle ^{\frac{1}{2}}\left\langle \left( A_{i}^{\ast
}g^{2}\left( \left\vert T_{i}^{\ast }\right\vert \right) A_{i}\right)
^{p}x,x\right\rangle ^{\frac{1}{2}}\\
 &\hspace{7.5cm}({\rm by\ Lemma}\ \ref{le2}(a) )\\
&\leq \underset{i=1}{\overset{n}{\sum }}\left( \frac{1}{2}\left(
\left\langle \left( B_{i}^{\ast }f^{2}\left( \left\vert T_{i}\right\vert
\right) B_{i}\right) ^{p}x,x\right\rangle ^{r}+\left\langle \left(
A_{i}^{\ast }g^{2}\left( \left\vert T_{i}^{\ast }\right\vert \right)
A_{i}\right) ^{p}x,x\right\rangle ^{r}\right) \right) ^{\frac{1}{r}} \\
&\hspace{7.5cm}({\rm by\ Lemma}\ \ref{le1}(a) )\\
&\leq \underset{i=1}{\overset{n}{\sum }}\left( \frac{1}{2}\left\langle
\left( \left( B_{i}^{\ast }f^{2}\left( \left\vert T_{i}\right\vert \right)
B_{i}\right) ^{rp}+\left( A_{i}^{\ast }g^{2}\left( \left\vert T_{i}^{\ast
}\right\vert \right) A_{i}\right) ^{rp}\right) x,x\right\rangle \right) ^{
\frac{1}{r}} \\
&\hspace{7.5cm}({\rm by\ Lemma}\ \ref{le2}(a) )\\
&\leq \left( \frac{1}{2}\left\langle \underset{i=1}{\overset{n}{\sum }}
\left( \left( \left( B_{i}^{\ast }f^{2}\left( \left\vert T_{i}\right\vert
\right) B_{i}\right) ^{rp}+\left( A_{i}^{\ast }g^{2}\left( \left\vert
T_{i}^{\ast }\right\vert \right) A_{i}\right) ^{rp}\right) \right)
x,x\right\rangle \right) ^{\frac{1}{r}}
\end{align*}
Thus
\begin{align*}
&\hspace{-0.5cm}\left( \underset{i=1}{\overset{n}{\sum }}\left\vert \left\langle A_{i}^{\ast}T_{i}B_{i}x,x\right\rangle \right\vert ^{p}\right)^{r}\\
&\leq \frac{1}{2}
\left\langle \left( \underset{i=1}{\overset{n}{\sum }}\left( \left(
B_{i}^{\ast }f^{2}\left( \left\vert T_{i}\right\vert \right) B_{i}\right)
^{rp}+\left( A_{i}^{\ast }g^{2}\left( \left\vert T_{i}^{\ast }\right\vert
\right) A_{i}\right) ^{rp}\right) \right) x,x\right\rangle
\end{align*}
Now the result follows by taking the supremum over all unit vectors in $\mathscr{H}$.
\end{proof}
%=================================================================================================%
Choosing $A=B=I$, we get.
\begin{corollary}
Let $\left( T_{1},\ldots ,T_{n}\right) \in \mathbb{B}(\mathscr{H})^{(n)}$ and let $f$ and $g$ be nonnegative continuous functions on $\left[ 0,\infty
\right) $ satisfying $f\left( t\right)
g\left( t\right) =t$ for all $t\in \left[ 0,\infty \right) $. Then%
\begin{equation*}
w_{p}^{rp}\left( T_{1},\ldots ,T_{n}\right) \leq \frac{1}{2}\left\Vert
\underset{i=1}{\overset{n}{\sum }}\left( f^{2rp}\left( \left\vert
T_{i}\right\vert \right) +g^{2rp}\left( \left\vert T_{i}^{\ast }\right\vert
\right) \right) \right\Vert
\end{equation*}%
for $p\geq 1$ and $r\geq 1.$
\end{corollary}
%=================================================================================================%
Letting $f\left( t\right) =g\left( t\right) =t^{\frac{1}{2}}$, we get.
\begin{corollary}
Let $\left( T_{1},\ldots ,T_{n}\right) ,\left( A_{1},\ldots ,A_{n}\right)
,\left( B_{1},\ldots ,B_{n}\right)$ are in $\mathbb{B}(\mathscr{H})^{(n)}$%
. Then%
\begin{equation*}
w_{p}^{rp}\left( A_{1}^{\ast }T_{1}B_{1},\ldots ,A_{n}^{\ast
}T_{n}B_{n}\right) \leq \frac{1}{2}\left\Vert \underset{i=1}{\overset{n}{\sum }}\Big(
 \left( B_{i}^{\ast }\left\vert
T_{i}\right\vert B_{i}\right) ^{rp}+\left( A_{i}^{\ast }\left\vert
T_{i}^{\ast }\right\vert A_{i}\right) ^{rp}\Big)\right\Vert
\end{equation*}%
for $p\geq 1$ and $r\geq 1.$
\end{corollary}
%=================================================================================================%
\begin{corollary}
Let $\left( A_{1},\ldots ,A_{n}\right) ,\left( B_{1},\ldots ,B_{n}\right)
\in \mathbb{B}(\mathscr{H})^{(n)}$. Then%
\begin{equation*}
w_{p}^{rp}\left( A_{1}^{\ast }B_{1},\ldots ,A_{n}^{\ast }B_{n}\right) \leq
\frac{1}{2}\left\Vert \underset{i=1}{\overset{n}{\sum }}\Big( \left\vert B_{i}\right\vert ^{2rp}+\left\vert
A_{i}\right\vert ^{2rp}\Big)\right\Vert
\end{equation*}%
for $p\geq 1$ and $r\geq 1.$
\end{corollary}
%=================================================================================================%
\begin{corollary}
Let $\left( T_{1},\ldots ,T_{n}\right) \in \mathbb{B}(\mathscr{H})^{(n)}$. Then%
\begin{equation*}
w_{p}^{p}\left( T_{1},\ldots ,T_{n}\right) \leq \frac{1}{2}\left\Vert
\underset{i=1}{\overset{n}{\sum }}\left( \left\vert T_{i}\right\vert
^{2\alpha p}+\left\vert T_{i}^{\ast }\right\vert ^{2\left( 1-\alpha \right)
p}\right) \right\Vert
\end{equation*}%
for $0\leq \alpha \leq 1,$ and $p\geq 1.$ In particular.%
\begin{equation*}
w_{p}^{p}\left( T_{1},\ldots ,T_{n}\right) \leq \frac{1}{2}\left\Vert
\underset{i=1}{\overset{n}{\sum }}\left( \left\vert T_{i}\right\vert
^{p}+\left\vert T_{i}^{\ast }\right\vert ^{p}\right) \right\Vert .
\end{equation*}
\end{corollary}
%=================================================================================================%
\begin{corollary}
Let $B,C\in \mathbb{B}(\mathscr{H}) $. Then%
\begin{equation*}
w_{p}^{p}\left( B,C\right) \leq \frac{1}{2}\left\Vert \ \left\vert
B\right\vert ^{2\alpha p}+\left\vert B^{\ast }\right\vert ^{2\left( 1-\alpha
\right) p}+\left\vert C\right\vert ^{2\alpha p}+\left\vert C^{\ast
}\right\vert ^{2\left( 1-\alpha \right) p} \right\Vert
\end{equation*}%
for $0\leq \alpha \leq 1,$ and $p\geq 1.$ In particular.%
\begin{equation*}
w_{p}^{p}\left( B,C\right) \leq \frac{1}{2}\left\Vert \left\vert
B\right\vert ^{p}+\left\vert B^{\ast }\right\vert ^{p}+\left\vert
C\right\vert ^{p}+\left\vert C^{\ast }\right\vert ^{p} \right\Vert .
\end{equation*}
\end{corollary}
%=================================================================================================%
The next results are related to some different upper bounds for $w_p$ for $n$-tuples of operators, which have several inequalities as special cases.
\begin{proposition}
Let $\left( T_{1},\ldots ,T_{n}\right) \in \mathbb{B}(\mathscr{H})^{(n)} .$ Then%
\begin{equation*}
w_{p}\left( T_{1},\ldots ,T_{n}\right) \leq \frac{1}{2}\left\Vert
\underset{i=1}{\overset{n}{\sum }}\left( \left\vert T_{i}\right\vert
^{2\alpha }+\left\vert T_{i}^{\ast }\right\vert ^{2\left( 1-\alpha \right)
}\right) ^{p} \right\Vert ^{\frac{1}{p}}
\end{equation*}%
for $0\leq \alpha \leq 1,$ and $p\geq 1.$
\end{proposition}
\begin{proof}
By using the arithmetic-geometric mean, for any unit vector $x\in\mathscr{H}$ we have
\begin{align*}
\underset{i=1}{\overset{n}{\sum }}\left\vert \left\langle
T_{i}x,x\right\rangle \right\vert ^{p} &\leq \underset{i=1}{\overset{n}{%
\sum }}\left( \left\langle \left\vert T_{i}\right\vert ^{2\alpha
}x,x\right\rangle ^{\frac{1}{2}}\left\langle \left\vert T_{i}^{\ast
}\right\vert ^{2\left( 1-\alpha \right) }x,x\right\rangle ^{\frac{1}{2}%
}\right) ^{p}\\
&\hspace{6cm}({\rm by\ Lemma}\ \ref{le1}(b) )\\
&\leq \frac{1}{2^{p}}\underset{i=1}{\overset{n}{\sum }}\left( \left\langle
\left\vert T_{i}\right\vert ^{2\alpha }x,x\right\rangle +\left\langle
\left\vert T_{i}^{\ast }\right\vert ^{2\left( 1-\alpha \right)
}x,x\right\rangle \right) ^{p} \\
&=\frac{1}{2^{p}}\underset{i=1}{\overset{n}{\sum }}\left\langle \left(
\left\vert T_{i}\right\vert ^{2\alpha }+\left\vert T_{i}^{\ast }\right\vert
^{2\left( 1-\alpha \right) }\right) x,x\right\rangle ^{p}. \\
&\leq \frac{1}{2^{p}}\underset{i=1}{\overset{n}{\sum }}\left\langle \left(
\left\vert T_{i}\right\vert ^{2\alpha }+\left\vert T_{i}^{\ast }\right\vert
^{2\left( 1-\alpha \right) }\right) ^{p}x,x\right\rangle\\
&\hspace{6cm}({\rm by\ Lemma}\ \ref{le2}(a) )
\end{align*}
Now the result follows by taking the supremum over all unit vectors in $\mathscr{H}$.
\end{proof}
%=================================================================================================%
\begin{proposition}
Let $\left( T_{1},\ldots ,T_{n}\right) \in \mathbb{B}(\mathscr{H})^{(n)}$. Then%
\begin{equation*}
w_{p}\left( T_{1},\ldots ,T_{n}\right) \leq \left\Vert \underset{i=1}{%
\overset{n}{\sum }}\left( \alpha \left\vert T_{i}\right\vert ^{p}+\left(
1-\alpha \right) \left\vert T_{i}^{\ast }\right\vert ^{p}\right) \right\Vert
^{\frac{1}{p}}
\end{equation*}%
for $0\leq \alpha \leq 1,$ and $p\geq 2.$
\end{proposition}
\begin{proof}
For every unit vector \ $x\in \mathscr{H}$, we have
\begin{align*}
&\hspace{-0.2cm}\underset{i=1}{\overset{n}{\sum }}\left\vert \left\langle
T_{i}x,x\right\rangle \right\vert ^{p}\\&=\underset{i=1}{\overset{n}{\sum }}
\left( \left\vert \left\langle T_{i}x,x\right\rangle \right\vert ^{2}\right)
^{\frac{p}{2}} \\
&\leq \underset{i=1}{\overset{n}{\sum }}\left( \left\langle \left\vert
T_{i}\right\vert ^{2\alpha }x,x\right\rangle \left\langle \left\vert
T_{i}^{\ast }\right\vert ^{2\left( 1-\alpha \right) }x,x\right\rangle
\right) ^{\frac{p}{2}} \hspace{.5cm}({\rm by\ Lemma}\ \ref{le1}(b) )\\
&\leq \underset{i=1}{\overset{n}{\sum }}\left\langle \left\vert
T_{i}\right\vert ^{\alpha p}x,x\right\rangle \left\langle \left\vert
T_{i}^{\ast }\right\vert ^{\left( 1-\alpha \right) p}x,x\right\rangle
\hspace{1.3cm}({\rm by\ Lemma}\ \ref{le2}(a) ) \\
&\leq \underset{i=1}{\overset{n}{\sum }}\left\langle \left\vert
T_{i}\right\vert ^{p}x,x\right\rangle ^{\alpha }\left\langle \left\vert
T_{i}^{\ast }\right\vert ^{p}x,x\right\rangle ^{\left( 1-\alpha \right) }
\hspace{1.4cm}({\rm by\ Lemma}\ \ref{le2}(b) ) \\
&\leq \underset{i=1}{\overset{n}{\sum }}\Big( \alpha \left\langle
\left\vert T_{i}\right\vert ^{p}x,x\right\rangle +( 1-\alpha )
\left\langle \left\vert T_{i}^{\ast }\right\vert ^{p}x,x\right\rangle \Big)
\hspace{0cm}({\rm by\ Lemma}\ \ref{le1}(a) )\\
&\leq \underset{i=1}{\overset{n}{\sum }}\left\langle \Big( \alpha
\left\vert T_{i}\right\vert ^{p}+\left( 1-\alpha \right) \left\vert
T_{i}^{\ast }\right\vert ^{p}\Big) x,x\right\rangle \\
&=\left\langle \left( \underset{i=1}{\overset{n}{\sum }}\left( \alpha
\left\vert T_{i}\right\vert ^{p}+\left( 1-\alpha \right) \left\vert
T_{i}^{\ast }\right\vert ^{p}\right) \right) x,x\right\rangle .
\end{align*}%
Now the result follows by taking the supremum over all unit vectors in $\mathscr{H}$.
\end{proof}
%=================================================================================================%
\begin{remark}
As special cases,

(1) For $\alpha =\frac{1}{2}$, we have%
\begin{equation*}
w_{p}^{p}\left( T_{1},\ldots ,T_{n}\right) \leq \frac{1}{2}\left\Vert
\underset{i=1}{\overset{n}{\sum }}\left( \left\vert T_{i}\right\vert
^{p}+\left\vert T_{i}^{\ast }\right\vert ^{p}\right) \right\Vert .
\end{equation*}

(2) For $B,C\in \mathbb{B}(\mathscr{H}) ,0\leq \alpha \leq 1,$ and $p\geq 1,$ we have%
\begin{equation*}
w_{p}^{p}\left( B,C\right) \leq \left\Vert \alpha \left\vert B\right\vert
^{p}+\left( 1-\alpha \right) \left\vert B^{\ast }\right\vert ^{p}+\alpha
\left\vert C\right\vert ^{p}+\left( 1-\alpha \right) \left\vert C^{\ast
}\right\vert ^{p}\right\Vert .
\end{equation*}%
In particular,%
\begin{equation*}
w_{p}^{p}\left( B,C\right) \leq \frac{1}{2}\left\Vert \left\vert
B\right\vert ^{p}+\left\vert B^{\ast }\right\vert ^{p}+\left\vert
C\right\vert ^{p}+\left\vert C^{\ast }\right\vert ^{p}\right\Vert .
\end{equation*}
\end{remark}
%=================================================================================================%
The next result reads as follows.
\begin{proposition}
Let $\left( T_{1},\ldots ,T_{n}\right) \in \mathbb{B}(\mathscr{H})^{(n)}, 0\leq \alpha \leq 1,r\geq 1$ and $p\geq 1$. Then%
\begin{equation*}
w_{p}\left( T_{1},\ldots ,T_{n}\right) \leq \left( \underset{i=1}{\overset{n}%
{\sum }}\left\Vert \alpha \left\vert T_{i}\right\vert ^{2r}+\left( 1-\alpha
\right) \left\vert T_{i}^{\ast }\right\vert ^{2r}\right\Vert ^{\frac{p}{2r}%
}\right) ^{\frac{1}{p}}.
\end{equation*}
\end{proposition}
\begin{proof}
Let $x\in \mathscr{H}$ be a unit vector.
\begin{align*}
&\hspace{-0.2cm}\underset{i=1}{\overset{n}{\sum }}\left\vert \left\langle
T_{i}x,x\right\rangle \right\vert ^{p}\\
&=\underset{i=1}{\overset{n}{\sum }}%
\left( \left\vert \left\langle T_{i}x,x\right\rangle \right\vert ^{2}\right)
^{\frac{p}{2}} \\
&\leq \underset{i=1}{\overset{n}{\sum }}\left( \left\langle \left\vert
T_{i}\right\vert ^{2\alpha }x,x\right\rangle \left\langle \left\vert
T_{i}^{\ast }\right\vert ^{2\left( 1-\alpha \right) }x,x\right\rangle
\right) ^{\frac{p}{2}} \hspace{1cm}({\rm by\ Lemma}\ \ref{le1}(b) )\\
&\leq \underset{i=1}{\overset{n}{\sum }}\left( \left\langle \left\vert
T_{i}\right\vert ^{2}x,x\right\rangle ^{\alpha }\left\langle \left\vert
T_{i}^{\ast }\right\vert ^{2}x,x\right\rangle ^{\left( 1-\alpha \right)
}\right) ^{\frac{p}{2}} \hspace{1cm}({\rm by\ Lemma}\ \ref{le2}(b) )\\
&\leq \underset{i=1}{\overset{n}{\sum }}\left( \alpha \left\langle
\left\vert T_{i}\right\vert ^{2}x,x\right\rangle ^{r}+\left( 1-\alpha
\right) \left\langle \left\vert T_{i}^{\ast }\right\vert
^{2}x,x\right\rangle ^{r}\right) ^{\frac{p}{2r}}
\hspace{0cm}({\rm by\ Lemma}\ \ref{le1}(a) )\\
&\leq \underset{i=1}{\overset{n}{\sum }}\Big( \alpha \left\langle
\left\vert T_{i}\right\vert ^{2r}x,x\right\rangle +\left( 1-\alpha \right)
\left\langle \left\vert T_{i}^{\ast }\right\vert ^{2r}x,x\right\rangle
\Big) ^{\frac{p}{2r}} \hspace{0cm}({\rm by\ Lemma}\ \ref{le2}(a) )\\
&\leq \underset{i=1}{\overset{n}{\sum }}\left\langle \left( \alpha
\left\vert T_{i}\right\vert ^{2r}+\left( 1-\alpha \right) \left\vert
T_{i}^{\ast }\right\vert ^{2r}\right) x,x\right\rangle ^{\frac{p}{2r}}.
\end{align*}%
Now the result follows by taking the supremum over all unit vectors in $\mathscr{H}$.
\end{proof}
%=================================================================================================%
\begin{remark}
Some special cases can be stated as follows:

(1) For $\alpha =\frac{1}{2}$, we have
\begin{equation*}
w_{p}\left( T_{1},\ldots ,T_{n}\right) \leq \left( \frac{1}{2^{\frac{p}{2r}}}%
\underset{i=1}{\overset{n}{\sum }}\left\Vert \left\vert T_{i}\right\vert
^{2r}+\left\vert T_{i}^{\ast }\right\vert ^{2r}\right\Vert ^{\frac{p}{2r}%
}\right) ^{\frac{1}{p}}.
\end{equation*}

(2) For $B,C\in \mathbb{B}(\mathscr{H}) ,0\leq \alpha \leq 1,$ and $p\geq 1,$ we have%
\begin{align*}
&\hspace{-0.2cm}w_{p}\left( B,C\right)\\
&\leq \left( \left\Vert \alpha \left\vert B\right\vert
^{2r}+\left( 1-\alpha \right) \left\vert B^{\ast }\right\vert
^{2r}\right\Vert ^{\frac{p}{2r}}+\left\Vert \alpha \left\vert C\right\vert
^{2r}+\left( 1-\alpha \right) \left\vert C^{\ast }\right\vert
^{2r}\right\Vert ^{\frac{p}{2r}}\right) ^{\frac{1}{p}}.
\end{align*}
In particular,
\begin{equation*}
w_{p}\left( B,C\right) \leq \frac{1}{2^{\frac{1}{2r}}}\left( \left\Vert
\left\vert B\right\vert ^{2r}+\left\vert B^{\ast }\right\vert
^{2r}\right\Vert ^{\frac{p}{2r}}+\left\Vert \left\vert C\right\vert
^{2r}+\left\vert C^{\ast }\right\vert ^{2r}\right\Vert ^{\frac{p}{2r}%
}\right) ^{\frac{1}{p}}.
\end{equation*}
\end{remark}
\vspace{10 pt}
%=================================================================================================

\bigskip
\bibliographystyle{amsplain}

\end{document}